\documentclass[12pt,reqno]{amsart}

\usepackage{amssymb,latexsym}

\usepackage{enumerate}
\allowdisplaybreaks
\usepackage[french,english]{babel}
\usepackage{amsmath}
\usepackage{graphicx}
\usepackage{amssymb}
\usepackage{bbm}
\usepackage{amsthm,mathtools}
\usepackage{ulem}
\usepackage{geometry}
\usepackage{tikz-cd}
\usepackage{mathrsfs}
\usepackage[colorinlistoftodos]{todonotes}
\usepackage{enumitem}
\usepackage{verbatim}
\usepackage[foot]{amsaddr}
\usepackage{dsfont}
\usepackage{cite}
\usepackage[T1]{fontenc}

\makeatletter

\@namedef{subjclassname@2010}{
	
	\textup{2020} Mathematics Subject Classification}

\makeatother
\newtheorem{thm}{Theorem}[section]
\newtheorem*{thm*}{Theorem}

\newtheorem{lem}[thm]{Lemma}

\theoremstyle{definition}

\newtheorem{rem}{Remark}[section]

\numberwithin{equation}{section}

\newcommand{\inv}{^{-1}}

\newcommand{\mbz}{\mathbb{Z}}
\newcommand{\mbr}{\mathbb{R}}

\newcommand{\mce}{\mathcal{E}}

\newcommand{\mcp}{\mathcal{P}}

\usepackage{hyperref}
\hypersetup{colorlinks=true,linkcolor=blue,anchorcolor=blue,citecolor=blue}
\usepackage{color}
\newcommand{\newabstract}[1]{%
	\par\bigskip
	\csname otherlanguage*\endcsname{#1}%
	\csname captions#1\endcsname
	\item[\hskip\labelsep\scshape\abstractname.]
}

\begin{document}

	\baselineskip=17pt

	\title[Joint extreme values of the Riemann zeta function at harmonic points]{Joint extreme values of the Riemann zeta function at harmonic points}

	\author{Qiyu Yang\textsuperscript{1}}
    \author{Shengbo Zhao\textsuperscript{2}}
	\address{1.School of Mathematics and Statistics, Henan Normal University, Xinxiang 453007, CHINA}
	\address{2.School of Mathematical Sciences, Key Laboratory of Intelligent Computing and Applications(Ministry of Education), Tongji University, Shanghai 200092, P. R. China}
	\email{qyyang.must@gmail.com}
	\email{shengbozhao@hotmail.com}

	\begin{abstract} 
	   Using the resonance method, we obtain refined estimates for joint extreme values of the Riemann zeta function at harmonic points, improving upon Levinson’s 1972 results and providing new insight into the behavior of the Riemann zeta function. Our proof is primarily based on Dirichlet series theory and the truncated Euler product for the Riemann zeta function. As a corollary, we can
       recover some previously known extreme value results for the zeta function.
	\end{abstract}
	
	
	
	\maketitle

\section{Introduction}
The Riemann zeta function $\zeta(s)$ is one of the most important objects in analytic number theory, due to its close connection with the distribution of prime numbers. The study of the behavior of its extreme values has always been a valuable research topic.
\par
In 1972, Levinson's work \cite{levinson1972omega} refined the earlier results of Littlewood and showed that for sufficiently large $T$,  
$$
\max_{t \in [1,T]}|\zeta(1+it)| \ge e^\gamma \log_2 T + O(1),
$$
where $\gamma$ is the Euler-Mascheroni constant. Here and throughout, we write $\log_k$ for the $k$-th iterated logarithm. More than three decades later, Granville and Soundararajan \cite{Gran06extreme} further improved Levinson’s result by refining the lower bound and investigating the frequency of large values. They also proposed a conjecture regarding the true order of the maximum of $|\zeta(1+it)|$. Up to now, the sharpest lower bound was established by Aistleitner, Mahatab, and Munsch \cite{aistleitner2019extreme}. Using the long resonance method, they showed that for all sufficiently large $T$, there exists a constant $C$ such that

$$
\max_{t \in [\sqrt{T},T]} | \zeta(1+it)| \ge e^\gamma (\log_2 T + \log_3 T -C).
$$
\par
In the case of the real part $\sigma \in ( 1/2,1)$, Aistleitner \cite{aistleitner2016lower} refined earlier work of Montgomery \cite{montgomery1977extreme}, which relied on Diophantine approximation. Using the resonance method, Aistleitner \cite{aistleitner2016lower} proved that for $\sigma \in ( 1/2,1)$ fixed and sufficiently large $T$, 
$$
\max_{t \in[0,T]}|\zeta(\sigma+it)| \ge \exp \bigg(C_\sigma \frac{(\log T)^{1-\sigma}}{(\log_2 T)^\sigma} \bigg),
$$
where $C_\sigma = 0.18(2\sigma-1)^{1-\sigma}$. Subsequently, the celebrated work of Bondarenko and Seip \cite{bondarenko2018note} provided a thorough investigation into the behavior of extreme values of $\zeta(s)$ within the critical strip. Specifically, they showed that if $T$ is sufficiently large, then for $1/2 + 1/\log_2 T \le \sigma \le 3/4$,
$$
\max_{t \in [\sqrt{T},T]} |\zeta(\sigma+it)| \ge \exp \bigg(\nu(\sigma) \frac{(\log T)^{1-\sigma}}{(\log_2 T)^\sigma}\bigg),
$$
and for $3/4 \le \sigma \le 1-1/\log_2T$,
$$
\max_{t \in [T/2,T]} |\zeta(\sigma+it)| \ge \log_2T\exp \bigg(c+ \nu(\sigma) \frac{(\log T)^{1-\sigma}}{(\log_2 T)^\sigma}\bigg),
$$
with $c$ an absolute constant independent of $T$. Here, $\nu(\sigma)$ is a continuous function bounded below by $1/(2-2\sigma)$. Furthermore, it has the following asymptotic behavior:
$$
\nu(\sigma)=
\begin{cases}
(1-\sigma)^{-1}+O(|\log(1-\sigma)|), & \sigma \nearrow 1 \\
(1/\sqrt{2}+o(1))\sqrt{|\log(2\sigma-1)|}, & \sigma \searrow 1/2.
\end{cases}
$$
More recently, the first author of this paper unconditionally refined the work of Bondarenko and Seip \cite{bondarenko2018note} for any fixed $\sigma \in (0.78,1)$. Moreover, assuming the Riemann Hypothesis (RH), sharper lower bounds valid throughout the entire critical strip were obtained; see \cite{qiyu2024large}. In fact, the maximum of $|\zeta(\sigma+it)|$ was already predicted by Lamzouri based on a probabilistic model; see \cite[pp. 5453-5454]{lamzouri2011distribution}. Furthermore, extreme values of derivatives of $\zeta(s)$ on the $1$-line have also been established; see \cite{dong2023note,yang2022extreme}.
\par
In this paper, we are interested in joint extreme values of $\zeta(s)$ at harmonic points, namely at points of the form
$$
s=\sigma+it, \sigma+2it, \dots, \sigma+\ell it.
$$ 
Henceforth, we let $\ell \in \mbz^+$ be fixed, where $\mbz^+$ is the set of all positive integers. Through the refined estimation of the coefficients $d_\ell(n)$ of the Dirichlet series associated with $\zeta^\ell(s)$, Levinson \cite{levinson1972harmonic} proved that there exist arbitrarily large $T$ such that for $\ell \ge 2$,
$$
\prod_{j=1}^\ell |\zeta(1+ijt)| \ge (e^\gamma\log_2T)^\ell \Big(1+O\Big(\frac{\ell \log \ell}{\log_2T}\Big)\Big).
$$
While this result establishes the correct leading order for joint extreme values at harmonic points, it does not capture finer secondary terms. In particular, when $\ell = 1$, Aistleitner et al. \cite{aistleitner2019extreme} improved upon Levinson's result \cite{levinson1972omega} by incorporating additional logarithmic contributions. It is therefore natural to ask whether the lower bound can be further sharpened to reflect such secondary terms. Our first theorem provides such a refinement.
\begin{thm}
\label{thm1}
Let $T$ be sufficiently large. We have 
$$
\max_{t \in [\sqrt{T},T]}\prod_{j=1}^\ell |\zeta(1+ijt)| \ge e^{\ell \gamma} \big\{(\log_2 T)^\ell + \ell (\log_2 T)^{\ell -1}\log_3 T + O_\ell\big((\log_2 T)^{\ell -1}\big) \big\},
$$
where the $O$-term depends only on $\ell$.
\end{thm}
\begin{rem}
Here, we improve upon the above result of Levinson \cite{levinson1972harmonic} and obtain a secondary term $\ell(\log_2T)^{\ell-1}\log_3T$. It is worth noting that when $\ell=1$, the extreme value result coincides with that of Aistleitner et al. \cite{aistleitner2019extreme}.
\end{rem}
\par
The preceding theorem concerns joint extreme values on the line $\sigma = 1$. It is therefore natural to consider the corresponding problem in the critical strip, where the behavior of $\zeta(s)$ depends crucially on $\sigma$. The following theorem treats this case.
\begin{thm}
\label{thm2}
Let $\beta \in (0,1)$ be fixed, and let $T$ be sufficiently large. Then for $1/2 < \sigma < 1$, we have 
$$
\max_{t \in [T^\beta,T]}\prod_{j=1}^\ell |\zeta(\sigma+ijt)| \ge \exp\bigg\{ \kappa^{1-\sigma} \big(S(\sigma,\ell)+o_\sigma(1)\big)\frac{(\log T)^{1-\sigma}}{(\log_2 T)^\sigma}\bigg\}.
$$
Here, the $o$-term depends only on $\sigma$, and $S(\sigma,\ell)$ is an explicit constant defined by
\begin{equation*}
    S(\sigma,\ell) = \frac{\ell}{1-\sigma} + \sum_{m=1}^\ell (-1)^m \binom{\ell+1}{m+1} \frac{1}{1+\sigma(m-1)}.
\end{equation*}
And $\kappa$ is a positive real number satisfying 
$$
\kappa < \frac{1}{\sigma(1+c(\sigma))} \min \Big(1-\beta, \frac{\sigma-1/2}{7/4-\sigma/2}\Big),
$$
where $c(\sigma) \coloneqq \int_ 0^1 \mathrm{d}t/(2t^{-\sigma}-1)$.
\end{thm}

\begin{rem}
    In Theorem \ref{thm2}, when $\sigma$ is close to $1/2$, the growth of $S(\sigma, \ell)$ with respect to $\ell$ is significantly suppressed due to the strong cancellation produced by the alternating binomial sum involving $\ell/(1-\sigma)$. As $\sigma$ approaches $1$, this phenomenon gradually disappears, and $S(\sigma, \ell)$ then increases notably with $\ell$.
\end{rem}

\begin{rem}
In Theorems \ref{thm1} and \ref{thm2}, the ranges in which joint extreme values occur are $[\sqrt{T},T]$ and $[T^\beta,T]$, respectively, for sufficiently large $T$ and fixed $\beta \in (0,1)$. In fact, in Theorem \ref{thm1}, the range $[\sqrt{T},T]$ can be replaced by $[T^\beta,T]$, at the cost that the $O$-term would then depend on $\ell$ and $\beta$. Since the main term remains unchanged, we choose $\beta =1/2$ to simplify the presentation of the proof. In contrast, in Theorem \ref{thm2}, the parameter $\kappa$ appearing in the main term depends on $\beta$, and the effect of $\beta$ cannot be neglected.
\end{rem}

In Theorem \ref{thm2}, the parameter $\kappa$ is required to be bounded by two terms, where the second term comes from the assumption on the zero-free region given in Lemma \ref{lem2}. Under RH, the formula in Lemma \ref{lem4} is valid throughout $t\in[T^{\beta},T]$. Consequently, the conditional joint extreme values of $\zeta(\sigma+it)$ in the critical strip are derived.

\begin{thm}
\label{thm3}
    Assuming RH. Let $\beta \in (0,1)$ be fixed, and let $T$ be sufficiently large. Then for $1/2 < \sigma < 1$, we have 
    $$
    \max_{t \in [T^\beta,T]}\prod_{j=1}^\ell |\zeta(\sigma+ijt)| \ge \exp\bigg\{ \kappa^{1-\sigma} \big(S(\sigma,\ell)+o(1)\big)\frac{(\log T)^{1-\sigma}}{(\log_2 T)^\sigma}\bigg\}.
    $$
    Here $\kappa$ is a positive real number less than $(1-\beta)/(\sigma(1+c(\sigma))$.
\end{thm}
\begin{rem}
It is worth noting that, for fixed $\sigma \in (1/2,1)$, the range of conditional joint extreme values broadens as $\beta$ decreases, with the range of $\kappa$ increasing simultaneously. However, the order of magnitude in the exponent remains unchanged. This behavior is consistent with theoretical expectations. 
\end{rem}
\par
The study of joint extreme values at harmonic points reveals new structural features in the behavior of $\zeta(s)$. By examining the simultaneous growth of $\zeta(s)$ along homogeneous arithmetic progressions on the vertical line, we capture effects arising from the arithmetic structure of the Euler product. Specifically, the repeated occurrence of small primes $p$ across different $\zeta(\sigma+ijt)$ contributes significantly to this phenomenon. These results indicate that extreme values occur in a highly coordinated manner at different locations, leading to a deeper understanding of how extreme behavior propagates along harmonic directions. Moreover, the joint lower bounds extend the classical single-point phenomenon to a multi-point setting, thereby enriching extreme value theory and offering a more comprehensive description of the analytic behavior of $\zeta(s)$.
\par
It should be noted that $\ell$ in our main results is an arbitrary positive integer, in order to ensure that the truncation formulas \eqref{approximation1} and \eqref{approximation2} are valid.
\par
This paper is organized as follows. Section \ref{seclemma} presents some auxiliary lemmas. Theorems \ref{thm1} and \ref{thm2} are proved in Sections \ref{sec2} and \ref{sec3}, respectively. Finally, a short proof of Theorem \ref{thm3} is given in Section \ref{sec4}.

\section{Auxiliary lemmas}
\label{seclemma}
In this section, we introduce several lemmas. Define $\zeta(s;y)\coloneqq \prod_{p \le y}(1-p^{-s})^{-1}$. Now we state Lemma \ref{lem1}, which will be used to prove Theorem \ref{thm1}.
\begin{lem}[Aistleitner-Mahatab-Munsch \cite{aistleitner2019extreme}]
\label{lem1}
Let $T$ be sufficiently large, and set $Y = \exp \big( (\log T)^{10}\big)$. Then for $T^{1/10} \le |t| \le T$, we have
$$
\zeta(1+it) = \zeta(1+it;Y)\Big(1+O\Big(\frac{1}{\log T}\Big) \Big).
$$
\end{lem}

The following lemma, established by Granville and Soundararajan \cite{Gran06extreme}, provides an important result concerning $\log \zeta(\sigma+it)$.

\begin{lem}[Granville-Soundararajan \cite{Gran06extreme}]
    \label{lem2}
    Let $y \ge 2$ and $|t| \ge y+3$ be real numbers. Let $1/2 \le \sigma_0 <1$ and suppose that the rectangle $\{s: \sigma_0 < \operatorname{Re}(s) \le 1, |\operatorname{Im}(s)-t|\le y+2 \}$ is free of zeros of $\zeta(s)$. Then for $\sigma_0 < \sigma \le 1$, we have 
    $$
    \log \zeta(\sigma+it) = \sum_{n=2}^y \frac{\Lambda(n)}{n^{\sigma+it}\log n}+ O\bigg(\frac{\log|t|}{|\sigma_1 - \sigma_0|^2}y^{\sigma_1-\sigma}\bigg),
    $$
    where $\sigma_1 =\min (\sigma_0 + (\log y)\inv,(\sigma+\sigma_0)/2)$.
\end{lem}

Note that Lemma \ref{lem2} requires the assumption that there are no zeros of $\zeta(s)$ in a certain range. Therefore, in order to apply this result effectively, it is necessary to combine it with the zero-density result for $\zeta(s)$. Let $N(\sigma,T)$ denote the number of zeros of $\zeta(s)$ within the rectangle $\{s: \operatorname{Re}(s) \ge \sigma, 0< \operatorname{Im}(s) \ge T \}$. The following classical result can be found in \cite{Ingham1940}.
\begin{lem}[Ingham \cite{Ingham1940}]
    \label{lem3}
    Let $1/2 \le \sigma \le 1$, we have
    $$
    N(\sigma,T) \ll T^{\frac{3(1-\sigma)}{2-\sigma}}(\log T)^5.
    $$
\end{lem}

From Lemmas \ref{lem2} and \ref{lem3}, we obtain Lemma \ref{lem4} below, which provides a truncated formula for $\zeta(s)$ in the critical strip and plays an important role in the proof of Theorem \ref{thm2}.

\begin{lem}
    \label{lem4}
    Fix $\beta \in (0,1)$ Let $T$ be sufficiently large, and set $Y= (\log T)^{2026/(\sigma-1/2)}$. Then the truncated formula
    $$
    \zeta(\sigma+it) = \zeta(\sigma+it;Y)\Big(1+O\Big(\frac{1}{\log T}\Big) \Big)
    $$
    holds for all $t \in [T^\beta,T]\setminus\mce$. Here, we note that
    $$
    \operatorname{meas}(\mce) \ll T^{\frac{9/4-3\sigma/2}{7/4-\sigma/2}+o(1)}.
    $$
\end{lem}
\begin{proof}
   Choose $y=Y= (\log T)^{2026/(\sigma-1/2)}$ and $\sigma_0 = \sigma/2+1/4$ in Lemma \ref{lem2}. Meanwhile, using Lemma \ref{lem3} and we have
   $$
    \log \zeta(\sigma+it) = \sum_{n=2}^Y \frac{\Lambda(n)}{n^{\sigma+it}\log n}+ O\Big(\frac{1}{(\log T)^{10}}\Big)
   $$
for all $t \in [T^\beta,T]\setminus\mce$ with $\operatorname{meas}(\mce) \ll YT^{3(1-\sigma_0)/(2-\sigma_0)+o(1)}$. By the definition of $\Lambda(n)$, setting $n=p^k$ for $k \ge 1$ yields 
   $$
   \sum_{n=2}^Y \frac{\Lambda(n)}{n^{\sigma+it}\log n} = \sum_{p \le Y}\frac{1}{p^{\sigma+it}}+ \sum_{k \ge 2}\sum_{p^k \le Y}\frac{1}{kp^{k(\sigma+it)}}.
   $$
   This is because $\sigma \in (1/2,1)$ ensures that the terms with $k \ge 2$ are bounded.
   On the other hand, it is clear that
   $$
   \zeta(\sigma+it;Y) = \exp\Big(-\sum_{p \le Y}\log \Big(1-\frac{1}{p^{\sigma+it}}\Big) \Big) = \exp\Big(\sum_{p \le Y}\frac{1}{p^{\sigma+it}} +\sum_{k \ge 2}\sum_{p \le Y}\frac{1}{kp^{k(\sigma+it)}}\Big).
   $$
We observe that the sum
$$
\sum_{k \ge 2}\sum_{Y^{1/k} < p \le Y} \frac{1}{k p^{k\sigma}} = O \Big(\frac{1}{(\log T)^{10}} \Big)
$$
is sufficiently small by our choice of $Y$, so that $\zeta(\sigma+it)$ can be approximated well by $\zeta(\sigma+it;Y)$. This completes the proof of Lemma \ref{lem4}.
\end{proof}

\section{Proof of Theorem \ref{thm1}}
\label{sec2}
In this section, we shall present the proof of Theorem \ref{thm1} by applying the long resonance method developed in \cite{aistleitner2019extreme}.
 Set $Y = \exp \big( (\log T)^{10}\big)$. Using Lemma \ref{lem1}, the following asymptotic formula
\begin{equation}
    \label{approximation1}
    \prod_{j=1}^\ell\zeta(1+ijt) = \prod_{j=1}^\ell\zeta(1+ijt;Y)\Big(1+O\Big(\frac{1}{\log T}\Big) \Big)
\end{equation}
holds for all $t \in [\sqrt{T},T]$. Subsequently, we set
$
X = (\log T \log_2 T)/6,
$
and define $r(n)$ to be a completely multiplicative function whose values at primes are given by

\begin{equation*}
    r(p) =
    \begin{cases}
        1-\frac{p}{X}, ~ &\operatorname{if} p \le X, \\
		0, ~ &\operatorname{if} p > X.
    \end{cases}
\end{equation*}
Define 
$$
R(t) = \prod_{p \le X}\big(1-r(p)p^{it}\big)^{-1} = \sum_{n=1}^{\infty}r(n)n^{it}.
$$
Note that 
$$
\log |R(t)| \le \sum_{p \le X} \log \Big|\frac{1}{1-r(p)}\Big| = \sum_{p \le X}\log \frac{X}{p}.
$$
Using the prime number theory and the choice of $X$, it is easy to derive that
\begin{equation}
    \label{Rupper1}
    |R(t)|^2 \le T^{1/3+o(1)}.
\end{equation}
Moreover, the definition of $Y$ implies that
\begin{equation}
    \label{zetaYupper1}
    \prod_{j=1}^\ell |\zeta(1+ijt;Y)| \ll (\log Y)^\ell \ll T^{o(1)}.
\end{equation}
\par
Let $\Phi(y) \coloneqq e^{-y^2/2}$. The Fourier transform $\widehat{\Phi}$ satisfies $\widehat{\Phi}(\xi)=\sqrt{2\pi}\Phi(\xi) >0$ for all $\xi \in \mbr$.  Due to the exponential decay of $\Phi$, all sums and integrals associated with $\Phi$ that appear below are absolutely convergent. We then define the following two quantities:
\begin{align*}
    M_1 &\, \coloneqq M_1(R,T) = \int_{\sqrt{T}}^T |R(t)|^2 \Phi\Big(\frac{\log T}{T}t \Big)\mathrm{d}t, \\
    M_2 &\, \coloneqq M_2(R,T) = \int_{\sqrt{T}}^T \prod_{j=1}^\ell \zeta(1+ijt;Y) |R(t)|^2 \Phi\Big(\frac{\log T}{T}t \Big)\mathrm{d}t.
\end{align*}
Trivially, we have $\max_{t \in [\sqrt{T},T]}\prod_{j=1}^\ell |\zeta(1+ijt;Y)| \ge |M_2|/M_1$. To use the positivity of the Fourier transform of $\Phi$, we extend the integrals to $\mbr$. Naturally, considering the tails, we have
$$
\Big| \int_{|t|\le \sqrt{T}}  \prod_{j=1}^\ell \zeta(1+ijt;Y) |R(t)|^2\Phi\Big(\frac{\log T}{T}t \Big)\mathrm{d}t\Big| \ll T^{5/6+o(1)}
$$
by applying \eqref{Rupper1} and \eqref{zetaYupper1}. Furthermore, the rapid decay of $\Phi$ implies 
$$
\Big| \int_{|t|\ge T }  \prod_{j=1}^\ell \zeta(1+ijt;Y) |R(t)|^2 \Phi\Big(\frac{\log T}{T}t \Big)\mathrm{d}t\Big| \ll 1.
$$
Consequently, 
$$
2M_2= \int_{-\infty}^\infty \prod_{j=1}^\ell \zeta(1+ijt;Y) |R(t)|^2 \Phi\Big(\frac{\log T}{T}t \Big)\mathrm{d}t +O\big(T^{5/6+o(1)}\big).
$$
Similarly, setting $I_1\coloneqq I_1(R,T) =  \int_{-\infty}^\infty |R(t)|^2 \Phi (t\log T/T )\mathrm{d}t$, we have $2M_1 = I_1 + O\big(T^{5/6+o(1)}\big)$. 
Moreover, it follows from \cite[Eq. (8)]{aistleitner2019extreme} that
\begin{equation}
    \label{I1lower1}
    I_1 =\sum_{m,n=1}^\infty \int_{-\infty}^\infty r(m)r(n) \Big(\frac{m}{n} \Big)^{it}\Phi\Big(\frac{\log T}{T}t \Big)\mathrm{d}t  \gg T^{1+o(1)}.
\end{equation}
Noticing that $\widehat{\Phi}$ is always positive, we obtain that for $Y>X$,
$$
\int_{-\infty}^\infty \prod_{j=1}^\ell \zeta(1+ijt;Y) |R(t)|^2 \Phi\Big(\frac{\log T}{T}t \Big)\mathrm{d}t \ge \int_{-\infty}^\infty \prod_{j=1}^\ell \zeta(1+ijt;X) |R(t)|^2 \Phi\Big(\frac{\log T}{T}t \Big)\mathrm{d}t.
$$
Denote by $I_2 \coloneqq I_2(R,T)$ the integral on the right-hand side above. Then combining with \eqref{I1lower1} yields
\begin{equation}
    \label{max1}
    \max_{t \in [\sqrt{T},T]}\prod_{j=1}^\ell |\zeta(1+ijt;Y)| \ge \frac{|I_2|}{I_1}+O\big(T^{-1/6+o(1)}\big).
\end{equation}
\par
According to the definition of $r(n)$, we can write $\zeta(1+ijt;X)=\sum_{k_j=1}^\infty a_{k_j} {k_j}^{-ijt}$. Here, $a_{k_j}=1/k_j$ if all prime factors of $k_j$ are not exceeding $X$, and $a_{k_j}=0$ otherwise.
Expanding $\prod_{j=1}^\ell \zeta(1+ijt;X)$ and $|R(t)|^2$ in $I_2$, we obtain
$$
I_2 = \int_{-\infty}^\infty \sum_{k_1,k_2,\dots, k_\ell=1}^\infty a_{k_1}a_{k_2}\cdots a_{k_\ell}k_1^{-it}k_2^{-2it}\cdots k_{\ell}^{-\ell it} \sum_{m,n=1}^\infty r(m)r(n) \Big(\frac{m}{n} \Big)^{it}\Phi\Big(\frac{\log T}{T}t \Big)\mathrm{d}t.
$$
Since the order of summations and integrals can be interchanged freely, we obtain
$$
I_2 = \sum_{k_1,k_2,\dots, k_\ell=1}^\infty a_{k_1}a_{k_2}\cdots a_{k_\ell}\sum_{m,n=1}^\infty r(m)r(n)\int_{-\infty}^\infty k_1^{-it}k_2^{-2it}\cdots k_{\ell}^{-\ell it}\Big(\frac{m}{n} \Big)^{it}\Phi\Big(\frac{\log T}{T}t \Big)\mathrm{d}t.
$$
Using the fact that $\widehat{\Phi}$ is always positive, we can directly discard many terms in the second sum. Specifically, 
\begin{align*}
    I_2 & \, \ge \sum_{k_1,k_2,\dots, k_\ell=1}^\infty a_{k_1}a_{k_2}\cdots a_{k_\ell} \sum_{n=1}^\infty \sum_{\substack{m \ge 1 \\ k_1 k_2^2 \cdots k_\ell^\ell \mid m}} r(m)r(n) \int_{-\infty}^\infty \Big(\frac{m}{k_1 k_2^2 \cdots k_\ell^\ell n} \Big)^{it}\Phi\Big(\frac{\log T}{T}t \Big)\mathrm{d}t \\
    & \, = \sum_{k_1,k_2,\dots, k_\ell=1}^\infty a_{k_1}a_{k_2}\cdots a_{k_\ell}  \sum_{d,n=1}^\infty r(k_1 k_2^2 \cdots k_\ell^\ell d)r(n)\int_{-\infty}^\infty\Big(\frac{d}{n} \Big)^{it}\Phi\Big(\frac{\log T}{T}t \Big)\mathrm{d}t.
\end{align*}
Then, we have the following lower bound for $I_2$:
\begin{align*}
    I_2 &\, \ge \sum_{k_1,k_2,\dots, k_\ell=1}^\infty a_{k_1}a_{k_2}\cdots a_{k_\ell} r(k_1)r(k_2)^2\cdots r(k_\ell)^\ell \sum_{d,n=1}^\infty r(d)r(n)\int_{-\infty}^\infty\Big(\frac{d}{n} \Big)^{it}\Phi\Big(\frac{\log T}{T}t \Big)\mathrm{d}t \\
    &\, = I_1 \sum_{k_1,k_2,\dots, k_\ell=1}^\infty a_{k_1}a_{k_2}\cdots a_{k_\ell} r(k_1)r(k_2)^2\cdots r(k_\ell)^\ell.
\end{align*}
This follows from the fact that $r(n)$ is completely multiplicative, together with the definition of $I_1$ \eqref{I1lower1}. Therefore, we obtain 
\begin{equation}
    \label{I2I1ratio1}
    \frac{I_2}{I_1} \ge \sum_{k_1,k_2,\dots, k_\ell=1}^\infty a_{k_1}a_{k_2}\cdots a_{k_\ell} r(k_1)r(k_2)^2\cdots r(k_\ell)^\ell = \prod_{j=1}^\ell \Big( \sum_{k_j=1}^\infty a_{k_j} r(k_j)^j \Big).
\end{equation}
\par
For the inner sum on the right-hand side of \eqref{I2I1ratio1}, by the definition of $a(n)$ and $r(n)$, the sum vanishes whenever a prime divisor $p$ of $n$ exceeds $X$. Therefore, it can be written as a short Euler product $\sum_{k_j=1}^\infty a_{k_j} r(k_j)^j = \prod_{p \le X}(1-r(p)^j p^{-1})^{-1}$. Next, 
\begin{align*}
    \frac{I_2}{I_1} &\, \ge \prod_{j=1}^\ell\prod_{p \le X}\Big(1-\frac{r(p)^j}{p} \Big)^{-1} = \prod_{j=1}^\ell\prod_{p \le X} \Big(\frac{p}{p-1} \cdot \frac{p-1}{p-r(p)^j}\Big)\\
    &\, = \Big\{\prod_{j=1}^\ell\prod_{p \le X}\frac{p}{p-1} \Big\}\cdot\Big\{ \prod_{j=1}^\ell\prod_{p \le X} \frac{p-1}{p-r(p)^j} \Big\} \eqqcolon \mcp_1 \cdot \mcp_2.
\end{align*}
Then, we need to evaluate $\mcp_1$ and $\mcp_2$, respectively. First, Mertens' theorem shows that
$$
\prod_{p \le X}\frac{p}{p-1} = e^\gamma \log X +O(1) = e^\gamma(\log_2 T + \log_3 T)+O(1).
$$
Thus, we can compute $\mcp_1$ directly and get 
\begin{align}
      \label{p1}
    \mcp_1 &\, = \prod_{j=1}^\ell\big(e^\gamma(\log_2 T + \log_3 T)+O(1)\big) \nonumber\\ 
    &\, = e^{\ell \gamma}\big((\log_2 T)^\ell + \ell(\log_2 T)^{\ell-1}\log_3 T +O\big((\log_2 T)^{\ell-2}(\log_3 T)^2\big) \big).
\end{align}
On the other hand, for $\mcp_2$, we will show the following effective lower bound:
\begin{equation}
    \label{p2}
    \mcp_2 \ge 1 + O\Big(\frac{1}{\log_2 T} \Big).
\end{equation}
In fact, it is clear that
$$
\prod_{p \le X} \frac{p-1}{p-r(p)^j}  = \exp \Big( \sum_{p \le X} \log\frac{p-1}{p-r(p)^j}  \Big) =  \exp \Big( - \sum_{p \le X} \log \Big(1+ \frac{1-r(p)^j}{p-1}\Big)\Big).
$$
Noting that $r(p)^j = (1-p/X)^j \ge 1- jp/X$ for all $p \le X$, we use the prime number theory and obtain 
$$
\prod_{p \le X} \frac{p-1}{p-r(p)^j} \ge \exp \Big( - \sum_{p \le X} \frac{j}{X}\cdot \frac{p}{p-1}\Big) = \exp \Big( (-j+o(1))\frac{1}{\log X}\Big).
$$
Since $X = (\log T \log_2 T)/6$, we have $1/\log X \to 0$ as $T \to \infty$. Therefore, 
$$
\prod_{j=1}^\ell\prod_{p \le X} \frac{p-1}{p-r(p)^j}\ge \exp \Big( \Big(-\sum_{j=1}^\ell j+o(1)\Big)\frac{1}{\log X}\Big) = 1 + O\Big(\frac{1}{\log_2 T} \Big).
$$
This shows that \eqref{p2} holds.
\par
As a result, we get
\begin{equation}
    \label{i2i1}
    \frac{I_2}{I_1} \ge e^{\ell \gamma}\big((\log_2 T)^\ell + \ell(\log_2 T)^{\ell-1}\log_3 T +O\big((\log_2 T)^{\ell-1}\big) \big)
\end{equation}
by using \eqref{I2I1ratio1}, \eqref{p1} and \eqref{p2}. Finally, by substituting \eqref{i2i1} into \eqref{max1} and combining with \eqref{approximation1}, we prove Theorem \ref{thm1}.

\section{Proof of Theorem \ref{thm2}}
\label{sec3}

In accordance with the approach of Section \ref{sec2}, we require an approximation formula similar to \eqref{approximation1}. To this end, we set $Y= (\log T)^{2026/(\sigma-1/2)}$. Lemma \ref{lem4} yields the following approximation formula
\begin{equation}
    \label{approximation2}
        \prod_{j=1}^\ell\zeta(\sigma+ijt) = \prod_{j=1}^\ell\zeta(\sigma+ijt;Y)\Big(1+O\Big(\frac{1}{\log T}\Big) \Big),
\end{equation}
which holds for all $t \in [T^\beta,T]\setminus \mce$. Here,  the measure of $\mce$ satisfies
\begin{equation}
    \label{measmce}
    \operatorname{meas}(\mce) \ll T^{\frac{9/4-3\sigma/2}{7/4-\sigma/2}+o(1)}.
\end{equation}
\par
We set $X = \kappa \log T \log_2 T$ with $k>0$ to be chosen later, then we let $r(n)$ be a completely multiplicative function with values at primes given by
\begin{equation*}
    r(p) =
    \begin{cases}
        1-(\frac{p}{X})^\sigma, ~ &\operatorname{if} p \le X, \\
		0, ~ &\operatorname{if} p > X.
    \end{cases}
\end{equation*}
Define the resonator
$$
R(t) = \prod_{p \le X}\big(1-r(p)p^{it}\big)^{-1} = \sum_{n=1}^{\infty}r(n)n^{it}.
$$
Combining the choice of $X$ and the prime number theory, we have
\begin{equation}
    \label{Rupper2}
    |R(t)|^2 \le T^{2\kappa\sigma+o(1)}.
\end{equation}
Note that
$$
\zeta(\sigma+ijt;Y)=\prod_{p \le Y}\big(1-p^{-(\sigma+ijt)} \big)^{-1} = \exp \Big(\sum_{p \le Y}\frac{1}{p^{\sigma+ijt}}+O(1) \Big).
$$
Thus, we obtain that
\begin{equation}
    \label{max2}
    \max_{t \in [T^\beta,T]}\prod_{j=1}^\ell |\zeta(\sigma+ijt;Y)| \gg \exp \Big( \max_{t \in [T^\beta,T]}\Big( \operatorname{Re}\sum_{j=1}^\ell  \sum_{p \le Y}\frac{1}{p^{\sigma+ijt}}\Big) \Big).
\end{equation}
Furthermore, it is easy to show
\begin{equation}
    \label{zetaYupper2}
    \sum_{j=1}^\ell \sum_{p \le Y}\frac{1}{p^{\sigma+ijt}} \ll T^{o(1)}.
\end{equation}
\par
Now, we define the following four integrals
\begin{align*}
    M_1 &\, \coloneqq M_1(R,T) = \int_{T^\beta}^T |R(t)|^2 \Phi\Big(\frac{\log T}{T}t \Big)\mathrm{d}t, \\
    M_2 &\, \coloneqq M_2(R,T) = \int_{T^\beta}^T \Big(\operatorname{Re}\sum_{j=1}^\ell \sum_{p \le Y}\frac{1}{p^{\sigma+ijt}}\Big) |R(t)|^2 \Phi\Big(\frac{\log T}{T}t \Big)\mathrm{d}t, \\
    I_1 &\, \coloneqq I_1(R,T) = \int_{-\infty}^\infty |R(t)|^2 \Phi\Big(\frac{\log T}{T}t \Big)\mathrm{d}t, \\
    I_2 &\, \coloneqq I_2(R,T) = \int_{-\infty}^\infty \Big(\operatorname{Re}\sum_{j=1}^\ell \sum_{p \le Y}\frac{1}{p^{\sigma+ijt}}\Big) |R(t)|^2 \Phi\Big(\frac{\log T}{T}t \Big)\mathrm{d}t.
\end{align*}
Here, $\Phi(y) = e^{-t^2/2}$ coincides with that in Section \ref{sec2}. Plainly, 
\begin{equation}
    \label{max3}
    \max_{t \in [T^\beta,T]}\Big( \operatorname{Re}\sum_{j=1}^\ell  \sum_{p \le Y}\frac{1}{p^{\sigma+ijt}}\Big) \ge \frac{M_2}{M_1}.
\end{equation}
Using \eqref{Rupper2} and \eqref{zetaYupper2}, we have
$$
\bigg|\int_{|t| < T^\beta} \Big(\operatorname{Re}\sum_{j=1}^\ell \sum_{p \le Y}\frac{1}{p^{\sigma+ijt}}\Big) |R(t)|^2 \Phi\Big(\frac{\log T}{T}t \Big)\mathrm{d}t \bigg| \le T^{\beta + 2\kappa\sigma+o(1)}.
$$
Meanwhile, as $\Phi$ decays rapidly, it follows that
$$
\bigg|\int_{|t|>T} \Big(\operatorname{Re}\sum_{j=1}^\ell \sum_{p \le Y}\frac{1}{p^{\sigma+ijt}}\Big) |R(t)|^2 \Phi\Big(\frac{\log T}{T}t \Big)\mathrm{d}t \bigg| \ll 1.
$$
The two upper bounds above imply $2M_2=I_2+O(T^{\beta + 2\kappa\sigma+o(1)})$. Similarly, we have $2M_1=I_1+O(T^{\beta + 2\kappa\sigma+o(1)})$. Moreover, we will use the lower bound for $I_1$. We find
\begin{equation}
    \label{I1lower2}
    I_1 =\sum_{m,n=1}^\infty \int_{-\infty}^\infty r(m)r(n) \Big(\frac{m}{n} \Big)^{it}\Phi\Big(\frac{\log T}{T}t \Big)\mathrm{d}t  \gg T^{\kappa\sigma(1-c(\sigma))+1+o(1)}
\end{equation}
in \cite[pp. 78-79]{Dong2022phd}, where $c(\sigma) \coloneqq \int_ 0^1 \mathrm{d}t/(2t^{-\sigma}-1)$.
Assuming 
\begin{equation}
    \label{kappa1}
    \beta + 2\kappa\sigma < \kappa\sigma(1-c(\sigma))+1,
\end{equation}
the error term is negligible. In this case, we obtain the following lower bound:
\begin{equation}
    \label{m2m1i2i1}
    \frac{M_2}{M_1} \ge \frac{I_2}{I_1}+O\big(T^{\beta+\kappa\sigma(1+c(\sigma))+1+o(1)} \big).
\end{equation}
\par
Next, we find an effective lower bound for $I_2/I_1$. As in the proof of Theorem \ref{thm1}, we frequently use the positivity of $\widehat{\Phi}$ and obtain
\begin{align*}
    I_2 &\, \ge \int_{-\infty}^\infty \sum_{j=1}^\ell \sum_{p \le X}\frac{1}{p^{\sigma+ijt}} \sum_{m,n=1}^\infty r(m)r(n)\Big(\frac{m}{n}\Big)^{it} \Phi\Big(\frac{\log T}{T}t \Big)\mathrm{d}t \\
    &\, = \sum_{j=1}^\ell \sum_{p \le X} \frac{1}{p^\sigma} \sum_{m,n=1}^\infty r(m)r(n) \int_{-\infty}^\infty\Big(\frac{m}{np^j}\Big)^{it}\Phi\Big(\frac{\log T}{T}t \Big)\mathrm{d}t \\
    &\, \ge \sum_{j=1}^\ell \sum_{p \le X} \frac{1}{p^\sigma}  \sum_{n=1}^\infty \sum_{\substack{m \ge 1 \\ p^j \mid m}} r(m)r(n)\int_{-\infty}^\infty\Big(\frac{m}{np^j}\Big)^{it}\Phi\Big(\frac{\log T}{T}t \Big)\mathrm{d}t \\
    &\, = \sum_{j=1}^\ell \sum_{p \le X} \frac{r(p)^j}{p^\sigma}\sum_{k,n=1}^\infty \int_{-\infty}^\infty r(k)r(n) \Big(\frac{k}{n} \Big)^{it}\Phi\Big(\frac{\log T}{T}t \Big)\mathrm{d}t.
\end{align*}
Here in the last step, we set $m=kp^j$ and apply the fact that $r(n)$ is completely multiplicative. We get the following estimate from \eqref{I1lower2}:
$$
\frac{I_2}{I_1} \ge \sum_{j=1}^\ell \sum_{p \le X} \frac{r(p)^j}{p^\sigma}.
$$
\par
The definition of $r(p)$ implies that 
\begin{align*}
   \sum_{j=1}^\ell \sum_{p \le X} \frac{r(p)^j}{p^\sigma} &= \, \sum_{j=1}^\ell \sum_{p \le X} \sum_{m=0}^j \binom{j}{m}(-1)^m X^{-\sigma m} p^{\sigma(m-1)} \\
    &= \, \ell \sum_{p \le X} \frac{1}{p^\sigma} + \sum_{m=1}^\ell (-1)^m  X^{-\sigma m} \sum_{p \le X}p^{\sigma(m-1)}\sum_{j=m}^\ell \binom{j}{m} \\
    &= \, \ell \sum_{p \le X} \frac{1}{p^\sigma} + \sum_{m=1}^\ell (-1)^m \binom{\ell +1}{m+1} X^{-\sigma m} \sum_{p \le X}p^{\sigma(m-1)}.
\end{align*}
Utilizing the prime number theory, we obtain\footnote{Here, it is easy to show that $S(\sigma,\ell)>0$ for any fixed positive integer $\ell$ and $1/2<\sigma<1$.
}
\begin{equation}
    \label{I2I1lower}
    \frac{I_2}{I_1} \ge \big(S(\sigma,\ell) +o(1) \big) \frac{X^{1-\sigma}}{\log X},
\end{equation}
where
\begin{align*}
S(\sigma,\ell) & \, = \frac{\ell}{1-\sigma} - \frac{(\ell+1)\ell}{2} + \sum_{m=2}^\ell(-1)^m \binom{\ell +1}{m+1} \frac{1}{1+\sigma(m-1)} \\
   & \, = \frac{\ell}{1-\sigma} + \sum_{m=1}^\ell (-1)^m \binom{\ell+1}{m+1} \frac{1}{1+\sigma(m-1)}.
\end{align*}
\par
In order to apply \eqref{approximation2}, we also need to remove the contribution of the set $\mce$. To this end, by \eqref{measmce}, we may restrict $\kappa$ to satisfy
\begin{equation}
    \label{kappa2}
    2\kappa\sigma +\frac{9/4-3\sigma/2}{7/4-\sigma/2} < \kappa\sigma(1-c(\sigma)) +1.
\end{equation}
Therefore, combining \eqref{kappa1} and \eqref{kappa2}, we show $\kappa$ satisfies that
\begin{equation*}
\kappa < \frac{1}{\sigma(1+c(\sigma))} \min \Big(1-\beta, \frac{\sigma-1/2}{7/4-\sigma/2}\Big).
\end{equation*}
Since $X = \kappa \log T \log_2 T$, we complete the proof of Theorem \ref{thm2} by using \eqref{approximation2}, \eqref{max2}, \eqref{max3}, \eqref{m2m1i2i1} and \eqref{I2I1lower}.

\begin{rem}
    In the proof of Theorem \ref{thm2}, in order to obtain an effective approximate formula, we choose $\sigma_0 = \sigma/2+1/4$ in Lemma \ref{lem2}. However, we also note that Dong got the same formula by taking $\sigma_0 = (1-\lambda)\sigma+\lambda/2$ for $\lambda \in (0,1)$, up to the upper bound for the measure of $\mce$, see \cite[Lemma 4.6.1]{Dong2022phd}. In fact, under Dong's choice, one may let the parameter $\lambda \to 0$ so that $\sigma_0$ is sufficiently close to $\sigma$. By Lemma \ref{lem3}, adopting this choice allows for an extremely slight refinement of the range of $\kappa$. But for the sake of simplicity, we do not pursue this approach.
\end{rem}

\begin{rem}
   When $\sigma$ is close to the $1$-line, Ingham's Lemma \ref{lem3} can be improved. For example, see \cite[Chapter 11.4]{ivic2012riemann} for details. By employing these more refined zero-density results, one can slightly enlarge the range of $\kappa$.

\end{rem}

\section{A short proof for Theorem \ref{thm3}}
\label{sec4}
To establish Theorem \ref{thm3}, we merely need to slightly modify the proof of Theorem \ref{thm2}. Take $Y= (\log T)^{2026/(\sigma-1/2)}$. By Lemma \ref{lem2}, we have
$$
   \prod_{j=1}^\ell\zeta(\sigma+ijt) = \prod_{j=1}^\ell\zeta(\sigma+ijt;Y)\Big(1+O\Big(\frac{1}{\log T}\Big) \Big)
$$
for all $t \in [T^\beta,T]$. Assuming RH, there are no zeros of $\zeta(s)$ in the region $\sigma > 1/2$, which removes the contribution of $\mce$ in Section \ref{sec3}. The remaining argument is essentially the same, and we therefore omit the details.

\section{Final comments}
Levinson \cite{levinson1972harmonic} noted that joint extreme values for 
$\prod_{j=1}^\ell \zeta(1/2+ijt)$ and $\prod_{j=1}^\ell 1/\zeta(1+ijt)$ can also be obtained.
 However, the long resonance method employed in this paper is no longer applicable in that setting. 
Since $1/\zeta(s) = \sum_{n \ge 1} \mu(n) n^{-s}$, ``positivity'' is lost due to the presence of the Möbius function $\mu(n)$, a property crucial in discarding many terms when establishing lower bounds. The method described in this paper is not applicable when $\sigma = 1/2$, as there are infinitely many zeros on the critical line. To overcome this, a convolution formula analogous to that in \cite{bondarenko2018argument} would need to be established, along with a computation of a more elaborate resonator $R(t)$. In addition, the maximum of $|\zeta(1/2+it)|$ was predicted earlier by Farmer et al. \cite{farmer2007maximum}. This problem is left for future work. 
\par
 Following the work of Aistleitner et al. \cite{aistleitner2019extreme}, Dixit and Mahatab \cite{dixit2021large} also studied lower bounds for a general family of $L$-functions on the $1$-line. Hence, it is possible to apply our method to obtain analogous results for these $L$-functions.

\section*{Acknowledgments}
Qiyu Yang was supported by the Natural Science Foundation of Henan Province (Grant No. 252300421782).
	
	\bibliographystyle{siam}
    \bibliography{reference}
\end{document}